\documentclass[a4paper,12pt]{article}

\usepackage[active]{srcltx}
\usepackage{verbatim}
\usepackage{amsmath}
\usepackage{amssymb}
\usepackage{url}
\usepackage{cite}
\usepackage{enumerate}

\newtheorem{theorem}{Theorem}
\newtheorem{lemma}{Lemma}

\newenvironment{proof}{{\noindent\bf Proof.}}{\hfill$\Box$\\}

\DeclareMathOperator{\Aut}{Aut}

\newcommand{\lng}{\langle}
\newcommand{\rng}{\rangle}

\newcommand{\R}{\mathbb R}

\newcommand{\tp}{^\top}

\begin{document}

\title{Positive operators of Extended Lorentz cones
}
\author{S. Z. N\'emeth\\School of Mathematics, University of Birmingham\\Watson Building, Edgbaston\\Birmingham B15 2TT, United Kingdom\\email: s.nemeth@bham.ac.uk
\and G. Zhang\\School of Mathematics, University of Birmingham\\Watson Building, Edgbaston\\Birmingham B15 2TT, United Kingdom\\email: gxz245@bham.ac.uk}
\date{}
\maketitle

\begin{abstract}
In this paper necessary conditions and sufficient conditions are given for a linear operator to be a  positive operator of an Extended Lorentz 
cone. Similarities and differences with the positive operators of Lorentz cones are investigated.  
\end{abstract}

\section{Introduction}
The Lorentz cone (second-oreder cone) is a very important cone in optimization problems. Many models in robust optimization, plant location problems and investment portfolio manangement can be formulated as a second-order cone programming \cite{AG2003}. In \cite{NZ20151} and \cite{NZ2016}, we generalized the $n$-dimensional Lorentz cone to mutually dual $p+q$-dimension cones:
\begin{equation}
L(p,q) = \{(x,u) \in \mathbb{R}^p \times \mathbb{R}^q : x \geq \|u\|e  \}
\end{equation}
and
\begin{equation}
M(p,q) = \{(x,u) \in \mathbb{R}^p \times \mathbb{R}^q : \langle x, e \rangle \geq \|u\|,x\ge0\}
\end{equation}
where  $e=(1, \dots, 1) \in \mathbb{R}^p $. The extended Lorentz cones $L(p, q)$ and $M(p,q)$ become Lorentz cones exactly in the special case 
$p=1$. This is the only case when $L(p, q)$ ($M(p,q)$) is self-dual.

The set $\Gamma(C)$ of positive operators of a cone $C$ is defined by \[\Gamma (C) = \{A \in \mathbb{R}^{(p+q)(p+q)} : AC \subseteq C\}\mbox{ 
\cite{LS1975}}.\]
The set of positive operator is a cone in $\mathbb{R}^{n \times n}$ \cite{SV1970}. It can be easily checked that $A$ is a positive operator of $C$ 
if and only if $A\tp$ is a positive operator of $C^*$. The authors of \cite{SV1970} introduced the characteristic matrix of the Lorentz cones as $J_n= diag(-1,\dots,-1,1)$, and showed that the Lorentz cone can be represented as :
\begin{equation*}
L(p,1)= \{x \in \mathbb{R}^n : x\tp J_n x \geq 0 \text{ and }x_n \geq 0 \}.
\end{equation*}
Moreover, they proved the following theorem which characterizes a positive operator by a positive semidefiniteness condtion \cite{SV1970}:
\begin{theorem}
Let $A \in \mathbb{R}^{n \times n}$. If $A \in \Gamma(L(p,1)) \cup \Gamma(-L(p,1))$, then there exists a $\mu \geq 0$
\begin{equation}\label{lposi}
A\tp J_n A \succeq \mu J_n
\end{equation}
Conversely, if $rank(A) \neq 1$ and there is a $\mu \geq 0$ such that \eqref{lposi} holds, then
\begin{equation*}
A \in \Gamma(L(p,1)) \cup \Gamma(-L(p,1)).
\end{equation*}
\end{theorem}

Since $L(p, q)$ and $M(p,q)$ are extensions of the second-order cone, the problem of finding the positive operators of $L(p,q)$ and $M(p,q)$
arises naturally.  The aim of this short note is to find both the necessary conditions and sufficient conditions  for a linear operator to be a
positive operator of $L(p,q)$ or $M(p,q)$ and state the similarities and differences between the case $p=1$ and $p>1$. In \cite{RS2016} Sznajder
determines all automorphism operators of $L(p,q)$. In particular these operators are also positive operators of $L(p,q)$. This shows that
the problem of finding all positive operators of $L(p,q)$ (or $M(p,q)$) is a much more difficult problem than the one solved by Sznajder. 
Although this problem is still open, the present note presents some partial results, by finding necessary conditions and sufficient conditions 
for a linear operator to be a positive operator of $L(p,q)$ (or $M(p,q)$).

The structure of the paper is as follows. First we introduce some notations. Then, we will prove a lemma about the characterization of an 
extended Lorentz cone. Finally based on this lemma, we present necessariy conditions and sufficient conditions for a linear operator to be a 
positive operator of an extended Lorentz cone.

\section{Notations}
Denote by $\mathbb{N}$ the set of positive integers. Let $k,m,p,q\in\mathbb{N}$ and $\mathbb{R}^m$ be the $m$-dimensional real 
Euclidean vector space. 

Define the direct product space $\mathbb{R}^p \times \mathbb{R}^q$ as the set of all pairs of vectors $(x,u)$, where $x \in \mathbb{R}^p$ and $u
\in \mathbb{R}^q$. 

Identify the vectors of $\R^m$ by column vectors and consider the canonical inner product defined on $\R^m$ by 
$\R^m\times\R^m\ni (x,y)\mapsto \lng x,y\rng:=x\tp y\in\R$ with the induced norm $\R^m\ni x\mapsto \|x\|=\sqrt{\lng x,x\rng}\in\R$.

The inner product in 
$\mathbb{R}^p \times \mathbb{R}^q\equiv\R^{p+q}$ is given by
\begin{equation*}
((x,u),(y,v))\mapsto \lng x,y\rng+\lng u,v\rng.
\end{equation*}
A closed and convex set $K$ is said to be a cone if $K\cap(-K)=\{0\}$, and $\lambda v \in K$, $\forall\lambda \geq 0$ and $\forall v \in K$. A cone $K$ is said to be a proper cone if the interior of $K$ is nonempty. Note that both $L(p. q)$ and $M(p,q)$ are proper cones.
For any cone $K \subseteq \mathbb{R}^m$, its dual cone is defined by
\begin{equation*}
D:= K^* : =\{y \in \mathbb{R}^m: \lng x, y \rng \geq 0 \text{, } \forall x \in K\}.
\end{equation*}
We also have that $D^* = K$, so $D$ and $K$ are mutually dual.
The complementarity set of $K$ is defined by
\begin{equation*}
	C(K) = \{ (x,s)\in\R^m\times\R^m: x \in K,\mbox{ }s \in K^*\mbox{ and }\lng x, s \rng = 0 \}\mbox{ \cite{RS2016}}.
\end{equation*}
A matrix $A \in \mathbb{R}^{m \times m}$ is said to be \emph{Lyapunov-like} on $K$ if $\lng Ax, s \rng=0$ for all $(x, s) \in C(K)$. Such matrix (transformations) are also characterized by the condition
\begin{equation*}
e^{tA} \in \Aut(K) \text{ for all } t \in \mathbb{R},
\end{equation*}
where $\Aut(K)$ denotes the automorphism group of the cone $K$. Note that any automorphism is a positive operator.

\section{Main results}

First we need to introduce two lemmas
\begin{lemma}\label{munion}
The extended Lorentz cone $M(p,q)$ can be expressed as
\begin{equation}\label{eu}
	M(p,q) =\{ z= (x, u) \in \mathbb{R}^p \times \mathbb{R}^q : z\tp Jz \geq 0,\mbox{ } x \in \mathbb{R}^p_+\},
\end{equation}
where
\begin{equation*}
J=
\begin{bmatrix}
E & 0 \\
0 & -I
\end{bmatrix}
\end{equation*}
and $E$ is a $p \times p$ matrix with all entries $1$.
\end{lemma}
\begin{proof}
We have $z=(x,u)\in M(p,q)$ if and only if $z\tp J z \geq 0$ and $x\ge0$, or equivalently
\begin{equation*}
0\le z\tp Jz = x\tp Ex - u\tp u=\langle x, e \rangle^2-\|u\|^2
\end{equation*}
and $x\ge0$. Thus, $z=(x,u)\in M(p,q)$ if and only if $\langle x, e \rangle\ge\|u\|$ and $x\ge0$.
\end{proof}

The next theorem states necessary conditions for a linear operator to be a positive operator of $M(p,q)$.

\begin{theorem}[Necessary conditions for positive operators of $\boldsymbol{M(p,q)}$] Let $p>1$ and $q>0$ be integers. Let  $A\in\R^{(p+q)\times(p+q)}$. If $A$ is a 
	positive operator of $M(p,q)$, then the following holds:
	\begin{enumerate}[(i)]
		\item\label{i2} The first $p$ rows of $A$ (more precisely their transpose) are in $L(p,q)$.
		\item\label{ii2} The first $p$ columns of $A$ are in $M(p,q)$.
		\item\label{iii2} By adding any $i$-th column $i = 1, \dots,p$ to the linear combination of the columns $p+1, \dots, p+q$ with coefficients 
			$u_1, \dots , u_q$ such that the Euclidean norm of $u =(u_1, \dots, u_q)\tp$ is one, we obtain an element in $M(p,q)$.
                     \item\label{iv2} The sum of any i-th column $i= 1, \dots, p$ with any $(p+j)$-th column $j = 1, \dots, q$ is in $M(p,q)$.
                     \item\label{v2}If $A$ is $M(p,q)$-Lyapunov like, then $e^{tA}\in \Aut(M(p,q))$ and hence it is in particular a positive
			     operator of $M(p,q)$, for any $t \in \mathbb{R}$. 
	\end{enumerate}

\end{theorem}
\begin{proof}
\begin{enumerate}[(i)]
		\item\label{i3} Since $A$ is a positive operator of $M(p,q)$, the first $p$ entries of $Az$ are nonnegative for any 
			$z \in M(p,q)$. Hence, the inner product of $z$ and any row vector of the first $p$ rows is nonnegative. Therefore these
			row vectors must be in the dual cone of $M(p,q)$, that is, $L(p,q)$.
		\item\label{ii3} Based on the above argument, $A\tp$ is a positive operator of $L(p,q)$, hence (ii) follows similarly to (i). 
		\item\label{iii3} Let 
\begin{equation*}
\beta_i = \alpha_i + \sum_{j=1} ^q u_j \alpha_{p+j},
\end{equation*}
where $\alpha_t$ is the $t$-th column of $A$.
Then for any $z \in L(p,q)$
\begin{equation*}
\langle z, \beta_i \rangle = \langle z, \alpha_i \rangle + \sum_{j=1} ^q u_j \langle z, \alpha_{p+j} \rangle .
\end{equation*}
By the Cauchy-Schwarz inequality, we have:
\begin{equation*}
 \sqrt{ \sum_{j=1} ^q \langle z, \alpha_{p+j} \rangle ^2}=\sqrt{ \sum_{j=1} ^q \langle z, \alpha_{p+j} \rangle ^2} \sqrt{\sum_{j=1}^q u^2_i}\geq \sum_{j=1} ^q u_j \langle z, \alpha_{p+j} \rangle.
\end{equation*} 
So 
\begin{equation*}
\langle z, \beta_i \rangle \geq \langle z, \alpha_i \rangle- \sqrt{ \sum_{j=1} ^q \langle z, \alpha_{p+j} \rangle ^2}.
\end{equation*}
Since $A$ is a positive operator of $M(p,q)$, $A\tp$ is a positive operator of $L(p,q)$ and therefore $A\tp z \in L(p,q)$. 
Since $\lng z, \alpha_k \rng $ is the $k$-th entry of $A\tp z$, by the definition of $L(p, q)$, we have that the right hand side of the above 
inequality is nonnegative. Thus, $\beta_i \in M(p,q)$.
                     \item\label{iv3} Obviously, it is a special case of the above assertion.
                     \item\label{v3}See \cite{RS2016}.
	\end{enumerate}
\end{proof}
\begin{theorem}[A sufficient condition for positive operators of $\boldsymbol{M(p,q)}$] If there exists a $\lambda \geq 0$ such that $A\tp JA -
	\lambda J$ is positive semidefinite and the first $p$ lines of $A$ are in $L(p,q)$, then $A$ is a positive operator of $M(p,q)$.
\end{theorem}
\begin{proof}
Since the first $p$ rows of $A$ are in $L(p,q)$, the first $p$ entries of $Az$ are nonnegative for any $z\in M(p,q)$. Since $A\tp JA - \lambda J$ is
positive semidefinite, we have
\begin{equation*}
0 \leq z\tp(A\tp JA - \lambda J)z =(Az)\tp J (Az) - \lambda z\tp J z.
\end{equation*}
By lemma \ref{munion}, $\lambda z\tp J z \geq 0$. So $(Az)\tp J (Az) \geq 0$, hence by lemma \ref{munion}, $Az \in M(p,q)$. Thus, we have that $A$ 
is a positive operator of $M(p,q)$.
\end{proof}
Similar necessary conditions and sufficient conditions can be given for the positive operators of $L(p,q)$.

\section{Conclusion}
In this paper, we showed necessary conditions and sufficient conditions for positive operators of Extended Lorentz cones. Since the first $p$
entries (rather than the first entry only of vectors in $L(1,q)$) must all be postive, some extra conditions (such as the first $p$ lines are in
$L(p,q)$) are needed to ensure that $A$ is a positive operator when $A\tp JA - \lambda J$ is positive semidefinite.

In the future we wish to find a necessary and sufficient condition for a linear operator to be a positive operator of and extended Lorentz cone.

\bibliographystyle{plain}
\bibliography{Ref}

\begin{thebibliography}{1}

\bibitem{AG2003}
F.~Alizadeh and D.~Goldfarb.
\newblock Second-order cone programming.
\newblock {\em Math. Program.}, 95(1, Ser. B):3--51, 2003.
\newblock ISMP 2000, Part 3 (Atlanta, GA).

\bibitem{LS1975}
R.~Loewy and H.~Schneider.
\newblock Positive operators on the {$n$}-dimensional ice cream cone.
\newblock {\em Journal of Mathematical Analysis and Applications}, 49:375--392,
  1975.

\bibitem{NZ20151}
S.~Z. N{\'e}meth and G.~Zhang.
\newblock Extended {L}orentz cones and mixed complementarity problems.
\newblock {\em J. Global Optim.}, 62(3):443--457, 2015.

\bibitem{NZ2016}
S.~Z. N{\'e}meth and G.~Zhang.
\newblock Extended {L}orentz cones and variational inequalities on cylinders.
\newblock {\em J. Optim. Theory Appl.}, 168(3):756--768, 2016.

\bibitem{SV1970}
H.~Schneider and M.~Vidyasagar.
\newblock Cross-positive matrices.
\newblock {\em SIAM J. Numer. Anal.}, 7:508--519, 1970.

\bibitem{RS2016}
R.~Sznajder.
\newblock The {L}yapunov rank of extended second order cones.
\newblock {\em Journal of Global Optimization}, DOI: 10.1007/s10898-016-0445-1,
  2016.

\end{thebibliography}
\end{document}